\documentclass[12pt]{article}
\usepackage[english]{babel}
\usepackage[square,numbers]{natbib}
\usepackage{amscd,float}
\usepackage{latexsym,amsmath,amssymb,amsbsy, amsthm}

\newcommand{\R}{\mathbb R}
\renewcommand{\S}{\mathcal{S}}

\newcommand{\acal}{\mathcal{A}}

\newcommand{\comment}[1]{}

\numberwithin{equation}{section}

\renewcommand{\P}{\mathbb{P}}

\usepackage[colorlinks = TRUE, linkcolor = blue, citecolor = blue,
urlcolor = blue]{hyperref}
\usepackage{setspace}
\usepackage{lscape,fancyhdr,fancybox}
\usepackage{graphicx,epsfig, xy}
\usepackage{color}
\usepackage[hmarginratio=1:1, vmarginratio =5:5,
textheight=22cm,bindingoffset=1.5cm, textwidth=14.6cm]{geometry}
\usepackage{float}
\usepackage{graphicx}
\usepackage{caption}
\usepackage{subcaption}

\newtheorem{theorem}{Theorem}[section]
\newtheorem{lemma}{Lemma}[section]
\newtheorem{proposition}{Proposition}[section]

\numberwithin{equation}{section}

\newtheorem{remark}{Remark}[section]

\theoremstyle{definition}
\newtheorem{definition}{Definition}[section]

\usepackage{amsfonts}

\usepackage{txfonts,enumerate}

\newcommand{\pr}{\mathbb P}

\DeclareMathOperator{\E}{E}

\def\ind{\mathbf{1}}
\renewcommand{\l}{\left}
\renewcommand{\r}{\right}
\definecolor{iblue}{rgb}{0.1,0,0.75}
\definecolor{ired}{rgb}{0.9,0,0.1}

\newcommand{\beq}{\begin{eqnarray}}
\newcommand{\eeq}{\end{eqnarray}}
\newcommand{\ben}{\begin{eqnarray*}}
\newcommand{\een}{\end{eqnarray*}}

\title{A Bernstein type inequality for sums of selections from three dimensional arrays}
 \author{
\sc Debapratim Banerjee \footnote{Email:dban@wharton.upenn.edu} ~and~ Matteo Sordello \footnote{Email: sordello@wharton.upenn.edu}
 \\ \small Dept. of Statistics, University of Pennsylvania\\}
\begin{document}
 \maketitle
 
\begin{abstract}
\noindent
We consider the three dimensional array $\acal = \{a_{i,j,k}\}_{1\le i,j,k \le n}$, with $a_{i,j,k} \in [0,1]$, and the two random statistics $T_{1}:= \sum_{i=1}^n \sum_{j=1}^n a_{i,j,\sigma(i)}$ and $T_{2}:= \sum_{i=1}^{n} a_{i,\sigma(i),\pi(i)}$, where $\sigma$ and $\pi$ are chosen independently from the set of permutations of $\{1,2,\ldots,n  \}.$ These can be viewed as natural three dimensional generalizations of the statistic $T_{3}=\sum_{i=1}^{n} a_{i,\sigma(i)}$, considered by Hoeffding \cite{Hoe51}. Here we give Bernstein type concentration inequalities for $T_{1}$ and $T_{2}$ by extending the argument for concentration of $T_{3}$ by Chatterjee \cite{Cha05}.
\end{abstract}

 \section{Arrays and Concentration Inequalities}
Let $\acal = \{a_{i,j,k}\}_{1\le i,j,k \le n}$ be a three dimensional array with $a_{i,j,k} \in [0,1]$, and consider the following two statistics
\begin{equation}\label{def:stat}
T_{1}:= \sum_{i=1}^n \sum_{j=1}^n a_{i,j,\sigma(i)} \qquad\text{and} \qquad T_{2}:= \sum_{i=1}^{n} a_{i,\sigma(i),\pi(i)}
\end{equation}
where $\sigma$ and $\pi$ are chosen independently and uniformly from the set $S_n$ of permutations of $[n] = \{1,2,\ldots,n\}$. Our goal is to obtain Bernstein type tail bounds for the statistics $T_{1}$ and $T_{2}$. Statistics of these type have already been considered in literature; for example, when the dimension is two, the statistic
\[
T_{3}:=\sum_{i=1}^{n} a_{i,\sigma(i)}
\]
 where $\sigma$ is drawn uniformly from $S_{n}$ was studied by Hoeffding \cite{Hoe51}, who proved that, under certain conditions, it has an asymptotic normal distribution as $n$ goes to infinity. In fact, the special case when $a_{i,j}=c_{i}\cdot d_{j}$ dates back to the works of Wald and Wolfwitz \cite{WW44} and Noether \cite{Noe49}. Another example of the statistic $T_{3}$ is the Spearman's footrule, useful in non-parametric statistics, where $a_{i,j}=|i-j|.$ Statistics $T_{1}$ and $T_{2}$ can be viewed as natural generalizations of statistic $T_{3}$ in three dimensions. However, in this paper we are concerned about concentration inequalities for $T_{1}$ and $T_{2}$, and not on their asymptotic distribution. The concentration of $T_{3}$ was considered by Chatterjee \cite{Cha05} (page 52); specifically he obtained an elegant tail bound of Bernstein type.
\begin{theorem}\label{thm:Cha}
Let $\{a_{i,j}\}_{1\le i,j,\le n} \in [0,1]$ and $T_{3}$ be as above. Then for any $t \ge 0,$ 
$$\mathbb{P}\left(|T_{3}- E(T_{3})| \ge  t \right) \le 2\exp\left\{ -\frac{t^2}{4\E[T_3]+2t} \right\}$$
\end{theorem}

\noindent 
Chatterjee obtains this bound by the method of exchangeable pairs, and here we extend his method to obtain Bernstein type concentration inequalities for $T_{1}$ and $T_{2}$. 
\begin{theorem}\label{thm:ours}
If $T_1$ and $T_{2}$ are as defined in (\ref{def:stat}) and $\{a_{i,j,k}\}_{1\le i,j,k\le n} \in [0,1]$, then 
\begin{align}
\mathbb{P}\left(|T_{1} - \E[T_{1}]| \geq t\right) &\leq 2 \exp\left\{-\frac{t^2}{2n\cdot(2\E[T_{1}] + t)}\right\}  \label{eqn:bddsta1} \\[8pt]
 \mathbb{P}\left(\left|T_{2} - \E[T_{2}]\right| \geq t\right)  &\leq 2\exp\left\{ -\frac{(t-3 + O(1/n))^{2}}{12\E(T_{2}) + 18+ 6(1+O(1/n))(t-3)} \right\} \label{eqn:bddsta2} .
 \end{align}

\end{theorem} 

\noindent 
Concentrations of functions of random permutations have also been studied by Talagrand (Theorem 5.1) \cite{Tal}, Murray \cite{Murr} and McDiarmid \cite{Mac}. However, as mentioned in Chatterjee \cite{Cha05}, apart from Talagrand's Theorem 5.1 none of these results are able to give Bernstein type concentration inequalities as above.

\section{On the method of exchangeable pair} 
We first need to recall some notions on the theory of exchangeable pairs as used by Chatterjee \cite{Cha05}. 

\begin{definition}
Suppose $X$ is a random variable on the measure space $(\Omega, \mathcal{F},\mathbb{P})$ and $X'$ is another random variable defined on the same measure space. The pair $(X,X')$ is called an exchangeable pair if $(X,X')\stackrel{d}{=}(X',X)$.
\end{definition}

\noindent 
The method of exchangeable pairs exploits three useful functions:
\begin{itemize}
\item A function $F:\mathbb{R}^2 \to \mathbb{R}$, measurable and almost surely anti-symmetric, i.e. such that $F(X,X')=-F(X',X)$ almost surely.
\item The function $f: \R \to \R$ defined by $f(X) := \E\left[ F(X,X') \left| X \right.\right]$. This is a fundamental quantity in the the concentration inequality.
\item The function $v(X)$, that serves as a stochastic bound size of $f(X)$, and which is defined by
\begin{equation}\label{def:v(x)}
v(X) := \frac{1}{2} \E\left[|(f(X) - f(X'))\cdot F(X, X')|\ \big|\ X\right].
\end{equation}
\end{itemize}
The following lemma from Chatterjee \cite{Cha05} tells us how the concentration of $f(X)$ is governed by a bound on $v(X)$.
\begin{lemma}(Theorem 3.9 in \cite{Cha05})\label{lem:vx}
Suppose $(X,X')$ is an exchangeable pair and $F(X,X')$, $f(X)$ and $v(X)$ are defined as before, with $v(X)\le C+Bf(X)$ almost surely for some known fixed constants $B$ and $C$. Then 
\[
\mathbb{P}\left(|f(X)| \geq t\right) \leq 2 \exp\l\{-\frac{t^2}{2C + 2Bt}\r\} 
\]
\end{lemma}

\noindent 
The fundamental idea of the method of exchangeable pairs is to construct $F(X,X')$, $f(X)$ and $v(X)$ so that Lemma \ref{lem:vx} yields concentration for $f(X)$. One example to keep in mind of $F(X,X')$ is $c(X-X')$, where $c$ is a nonrandom constant.

\begin{remark}
Chatterjee (\cite{Cha06}, Theorem 1.5) has further proved under the hypotheses of Lemma \ref{lem:vx} that for the lower tail probabilities $\P(f(X) \leq -t)$ one has a genuinely Gaussian bound of the form $\exp\l(-\frac{t^2}{2C}\r)$. One can also show by a further modification of Chatterjee's method that there is a Gaussian bound for the lower tail probabilities of $T_2$ and $T_3$, but we do not pursue these bounds here.
\end{remark}

\section{Strategy of the Proofs}
For proving the concentration inequalities for $T_{1}$ and $T_{2}$, we use the following general strategy. At first we construct the statistics $T_{1}'$ and $T_{2}'$ by applying ``small" changes to $T_{1}$ and $T_{2}$, such that two properties hold. We require $(T_{j},T_{j}')$ to form an exchangeable pair and $\E\left[ T_{j}- T_{j}' \left| T_{j} \right. \right]$ to be somewhat close to $c\left(T_{j}-\E[T_{j}]\right)/n$ for each $j \in \{ 1,2 \}$. We then define the quantity $v(T_{j})$ as in the previous section and bound it in terms of $T_{j}- \E[T_{j}]$. Finally, we derive the concentration inequality for $\left|T_{j}-\E[T_{j}]\right|$ by applying Lemma \ref{lem:vx}. 

The construction of $T_{1}'$ is done by choosing two indexes $I_{1}, I_{2}$ independently and uniformly at random from $[n]$, and considering the permutation $(I_1, I_2)$ that interchanges the two indexes.
We then define the permutation $\sigma'= \sigma \circ (I_{1}, I_{2})$ and the statistic $T_{1}':= \sum_{i=1}^{n} \sum_{j=1}^{n} a_{i,j,\sigma'(i)}$, and we prove that $\E\left[ T_{1}- T_{1}' \left| T_{1} \right. \right] = 2\left(T_{1}-\E[T_{1}]\right)/n$. This is a similar procedure to the one in Chatterjee \cite{Cha05}.

The construction of $T_{2}'$ is not as simple. The main reason is that $\E[T_{2}]$ is a sum over three independent directions $i, j, k$, while, fixing $\sigma$ and $\pi$, $T_{2}$ is a sum over only one single direction. As a consequence, one might check that it is not possible to get $\E\left[T_{2}- T_{2}'\left| T_{2} \right.\right]$ close to $c\left( T_{2}- \E[T_{2}]  \right)/n$ by simply moving two indexes. Instead, one needs to move three indexes in a systematic way. We then choose $(I_{1},I_{2},I_{3})$ extracted uniformly \textit{without replacement} from $[n]$ and define the functions $\tau_{1,2}: [n]^3 \to \S_n$ such that $\tau_{1}(I_{1},I_{2},I_{3})=(I_{1},I_{2},I_{3})$ and $\tau_{2}(I_{1},I_{2},I_{3})=(I_{1},I_{3},I_{2})$. These are the only cyclic permutations which are not the identity. The permutations $\sigma',\pi'$ are defined as follows:   
$$(\sigma', \pi') = \l\{\begin{array}{ll}
\l(\sigma\circ\tau_1(I_1,I_2,I_3), \pi\circ\tau_2(I_1,I_2,I_3)\r) \quad &\text{with probability } \frac{1}{2} \\[2pt]
\l(\sigma\circ\tau_2(I_1,I_2,I_3), \pi\circ\tau_1(I_1,I_2,I_3)\r) \quad &\text{with probability } \frac{1}{2} \end{array}\r.$$
Note that this $\sigma'$ is different from the one defined in the construction of $T_1'$, but it will always be clear which one of the two we are considering. Finally we define 
\begin{equation*}
T_2' :=\sum_{i} a_{i, \sigma'(i),\pi'(i)}  .
\end{equation*}
For $\sigma'$ and $\pi'$ to be valid permutations one needs all the indexes $(I_{1},I_{2},I_{3})$ to be distinct or for all three to be the same. We only consider the case when $(I_{1},I_{2},I_{3})$ are all distinct for convenience, since the case when they are all the same does not affect the exchangeability of $T_2$ and $T_2'$ and it just gives a slight change in the result which is negligible as $n$ grows to infinity. It is important to note that one needs $\sigma'$ and $\pi'$ to be valid permutations in order for $T_{2}'$ to have the same distribution as $T_{2}$, necessary condition to have exchangeability.

\section{Proofs of the results}
\noindent 
To prove (\ref{eqn:bddsta1}), we exchange two pairs of one dimensional rows (i.e. with $n$ elements each) in the $j^{th}$ direction of the matrix $\acal$, and get
$$T_{1}' := \sum_{i=1}^n \sum_{j=1}^n a_{i,j,\sigma'(i)} = T_1 + \sum_{j=1}^n \l(a_{I_1,j,\sigma(I_2)} + a_{I_2,j,\sigma(I_1)} - a_{I_1,j,\sigma(I_1)} - a_{I_2,j,\sigma(I_2)}\r)$$
where $\sigma' = \sigma \circ (I_1, I_2)$ and $I_{1}, I_{2}$ are extracted independently and uniformly at random from $[n]$.
\begin{proposition}{\label{exch}}
$(T_{1}, T_{1}')$ is an exchangeable pair.
\end{proposition}
\begin{proof}
We have the identity $\P\l(T_{1} = x, T_{1}' = x'\r) = \E\l[\P\l(T_{1} = x, T_{1}' = x' | \sigma, I_1, I_2\r)\r]$ and, since $(T_{1}, T_{1}')$ is a deterministic function of $\sigma, I_1$ and $I_2$, we can write
$$\pr\l(T_{1} = x, T_{1}' = x' | \sigma, I_1, I_2\r) = \ind\l(\sum_{i=1}^n \sum_{j=1}^n a_{i,j,\sigma(i)} = x,\ \ \sum_{i=1}^n \sum_{j=1}^n a_{i,j,\sigma\circ(I_1,I_2)(i)} = x'\r)$$
where as usual $\ind(\cdot)$ is the indicator function of the event in brackets. One then has
\begin{align*}
\pr\l(T_{1} = x, T_{1}' = x'\r) &= \E\l[\sum_{\gamma_1 \in \mathcal{S}_n}\frac{1}{n!}\cdot \ind\l(\sum_{i=1}^n \sum_{j=1}^n a_{i,j,\gamma_1(i)} = x,\ \ \sum_{i=1}^n \sum_{j=1}^n a_{i,j,\gamma_1\circ(I_1,I_2)(i)} = x'\r)\r] \\
&=\E\l[\sum_{\gamma_2 \in \mathcal{S}_n}\frac{1}{n!}\cdot \ind\l(\sum_{i=1}^n \sum_{j=1}^n a_{i,j,\gamma_2\circ(I_1,I_2)(i)} = x,\ \ \sum_{i=1}^n \sum_{j=1}^n a_{i,j,\gamma_2(i)} = x'\r)\r] \\
&= \pr\l(T_{1}' = x, T_{1} = x'\r)
\end{align*}
where we just set $\gamma_2 = \gamma_1 \circ (I_1, I_2)$. Moreover, for each fixed pair $(I_1, I_2)$, summing over all possible $\gamma_1$ is equivalent to summing over all $\gamma_2$, since both sums are made over the whole $\S_n$. 
\end{proof}

\noindent
When we take the expectation of $T_{1} - T_{1}'$ with respect to $\sigma$, $I_1$ and $I_2$, conditional on $T_{1}$, we have
\begin{align*}
\E\l[T_{1} - T_{1}' | T_{1}\r] &= \E\l[\sum_{j=1}^n \l(a_{I_1,j,\sigma(I_1)} + a_{I_2,j,\sigma(I_2)} - a_{I_1,j,\sigma(I_2)} - a_{I_2,j,\sigma(I_1)}\r) |\ T_{1}\r] \\
&= \sum_{j=1}^n \frac{1}{n^2} \sum_{i=1}^n \sum_{k=1}^n \E\l[a_{i,j,\sigma(i)} + a_{k,j,\sigma(k)} - a_{i,j,\sigma(k)} - a_{k,j,\sigma(i)} |\ T_{1}\r] \\
&= \frac2n T_{1} - \frac2n \sum_{i=1}^n \sum_{j=1}^n \frac{a_{i,j,1} + ... + a_{i,j,n}}{n} = \frac2n (T_{1} - \E[T_{1}]).
\end{align*}
We define $F(T_{1}, T_{1}') =  \frac n2 (T_{1} - T_{1}')$ and $f(T_{1}) := \E\l[F(T_{1}, T_{1}')\ |\ T_{1}\r] = T_{1} - \E[T_{1}]$, then the stochastic bound $v(T_{1})$ satisfies
\begin{align*}
v(T_{1}) &= \frac n4 \E\l[(T_{1}-T_{1}')^2\ |\ T_{1}\r] = \frac n4 \E\l[\l(\sum_{j=1}^n (a_{I_1,j,\sigma(I_1)} + a_{I_2,j,\sigma(I_2)} - a_{I_1,j,\sigma(I_2)} - a_{I_2,j,\sigma(I_1)})\r)^2 \bigg|\ T_{1} \r] \\
&\leq \frac{n^2}{2} \E\l[\sum_{j=1}^n (a_{I_1,j,\sigma(I_1)} + a_{I_2,j,\sigma(I_2)} + a_{I_1,j,\sigma(I_2)} + a_{I_2,j,\sigma(I_1)}) \bigg|\ T_{1} \r] \\
&= \frac{n^2}{2} \cdot \frac{1}{n^2} \sum_{i=1}^n \sum_{k=1}^n \E\l[\sum_{j=1}^n \l( a_{i,j,\sigma(i)} + a_{k,j,\sigma(k)} + a_{i,j,\sigma(k)} + a_{k,j,\sigma(i)}\r) \bigg|\ T_{1}\r] \\
&= n\cdot T_{1} + n \sum_{i=1}^n \sum_{j=1}^n \frac{a_{i,j,1} + ... + a_{i,j,n}}{n}\ =\ n (T_{1} + \E[T_{1}])\ =\ n\l(f(T_{1}) + 2\cdot \E[T_{1}]\r).
\end{align*}
Here we used the observation that if $\alpha$ and $\beta$ are non-negative numbers bounded by D, then $(\alpha-\beta)^2 \leq D\cdot(\alpha + \beta)$.
The hypothesis of Lemma \ref{lem:vx} then hold with $B = n$ and $C = 2n\E[T_{1}]$, and Lemma \ref{lem:vx} gives us
$$\P\l(|T_{1} - \E[T_1]| \geq t\r) \leq 2 \exp\l\{-\frac{t^2}{2n\cdot(2\E[T_{1}] + t)}\r\}.$$
\hfill{$\square$}

\begin{remark}
As $n$ increases, this bound gets weaker. If we consider $t$ increasing faster than $n$, for example $t = n^{1+\lambda}$, with $\lambda > 0$, we get
$$\P\l(|T_{1} - \E[T_1]| \geq n^{1+\lambda}\r) \leq 2 \exp\l\{-\frac{n^{2+2\lambda}}{2n\cdot(2\E[T_{1}] + n^{1+\lambda})}\r\} \approx 2 \exp\l\{-\frac{n^{\lambda}}{2}\r\}$$
as $n$ grows. 
\end{remark}

\noindent 
Now to prove (\ref{eqn:bddsta2}), we need a more delicate argument.
For $k \in \{1,2,3 \}$ we define $C_{k}$ to be the set of all ordered tuples $ (I_{1},I_{2},I_{3}) \in [n]^{3}$ such that $\#\{ I_{1},I_{2},I_{3} \}=k$. 
As already mentioned before, we need $(I_{1},I_{2},I_{3}) \in \l\{C_1, C_{3}\r\}$ in order for $\tau_{1}(I_{1},I_{2},I_{3})$ and $\tau_{2}(I_{1},I_{2},I_{3})$ to be valid permutations. It is easy to see that in that case we have $\tau_{1}(I_{1},I_{2},I_{3})^{-1}= \tau_{2}(I_{1},I_{2},I_{3})$. On the contrary, when $(I_{1},I_{2},I_{3}) \in C_2$, the permutations $\sigma'$ and $\pi'$ are not well-defined.
With the choice of $(I_{1},I_{2},I_{3}) \in C_{3}$, and $(\sigma', \pi')$ defined as before, we have

\begin{align*}
T_2' &:= \sum_{i} a_{i, \sigma'(i),\pi'(i)} \\
&= T_2 - \sum_{j=1}^3 a_{I_j, \sigma(I_j), \pi(I_j)} + \left\{ 
\begin{array}{cc}
\sum_{j=1}^3 a_{I_j, \sigma\circ\tau_1(I_1,I_2,I_3)(I_j), \pi\circ\tau_2(I_1,I_2,I_3)(I_j)} \quad &\text{with prob. } \frac{1}{2} \\[4pt]
\sum_{j=1}^3 a_{I_j, \sigma\circ\tau_2(I_1,I_2,I_3)(I_j), \pi\circ\tau_1(I_1,I_2,I_3)(I_j)} \quad &\text{with prob. } \frac{1}{2}.
\end{array}
\right.
\end{align*}

\begin{proposition}\label{exchII}
$(T_2,T_2')$ forms an exchangeable pair.
\end{proposition}
\begin{proof}
We start on the same lines of Proposition \ref{exch}. We first write the equation $\P\l( T_2= x, T_2'= x' \r) = \E \l[ \P\l(T_2=x, T_2'=x'\l|\sigma,\pi, \sigma',\pi' \r.\r)\r]$ and, since $(T_2,T_2')$ is a function of $(\sigma,\pi,\sigma',\pi')$, we have
$$\P\l(T_2=x, T_2'=x' \l| \sigma,\pi,\sigma',\pi' \r.\r)=\ind\l(\sum_{i}a_{i,\sigma(i),\pi(i)}=x , \sum_{i} a_{i,\sigma'(i),\pi'(i)}= x'\r) .$$
We set $\gamma_3 = \gamma_1\circ\tau_{1}(I_{1},I_{2},I_{3}), \ \gamma_4 = \gamma_2\circ\tau_{2}(I_{1},I_{2},I_{3}), \ \gamma_5 = \gamma_1\circ\tau_{2}(I_{1},I_{2},I_{3})$ and $\gamma_6 = \gamma_2\circ\tau_{1}(I_{1},I_{2},I_{3})$, to get
\begin{align*} 
&\P\l( T_2= x, T_2'= x' \r) \\
&= \E \l[\sum_{\gamma_{1},\gamma_{2} \in \mathcal{S}_n^{2}} \frac{1}{(n!)^2} \l\{\frac12\cdot\ind\l(\sum_{i}a_{i,\gamma_{1}(i),\gamma_{2}(i)}=x, \sum_{i}a_{i, \gamma_{1}\circ \tau_{1}(I_{1},I_{2},I_{3})(i),\gamma_{2}\circ \tau_{2}(I_{1},I_{2},I_{3})(i)}=x'  \r)  \r.\r.\\
&\l.\l. \qquad + \frac12\cdot\ind\l(\sum_{i}a_{i,\gamma_{1}(i),\gamma_{2}(i)}=x, \sum_{i}a_{i, \gamma_{1}\circ \tau_{2}(I_{1},I_{2},I_{3})(i),\gamma_{2}\circ \tau_{1}(I_{1},I_{2},I_{3})(i)}=x'  \r)\r\}  \r]\\
&= \E \l[\sum_{\gamma_{3},\gamma_{4} \in \mathcal{S}_n^{2}} \frac{1}{2(n!)^2} \cdot\ind\l(\sum_{i}a_{i, \gamma_{3}\circ \tau_{2}(I_{1},I_{2},I_{3})(i),\gamma_{4}\circ \tau_{1}(I_{1},I_{2},I_{3})(i)}=x , \sum_{i}a_{i,\gamma_{3}(i),\gamma_{4}(i)}=x'  \r)  \r.\\
&\l. \qquad + \sum_{\gamma_{5},\gamma_{6} \in \mathcal{S}_n^{2}} \frac{1}{2(n!)^2}\cdot\ind \l(\sum_{i}a_{i, \gamma_{5}\circ \tau_{1}(I_{1},I_{2},I_{3})(i),\gamma_{6}\circ \tau_{2}(I_{1},I_{2},I_{3})(i)}=x , \sum_{i}a_{i,\gamma_{5}(i),\gamma_{6}(i)}=x' \r) \r]\\
&=\P\l(T_2=x', T_2'=x \r).
\end{align*}
We have used the facts that $\tau_{1}(I_{1},I_{2},I_{3})^{-1}=\tau_{2}(I_{1},I_{2},I_{3})$
and that for each tuple $(I_1,I_2,I_3)$ it is equivalent to sum over all the pairs $(\gamma_1, \gamma_2), (\gamma_3, \gamma_4)$ or $(\gamma_5, \gamma_6)$.
\end{proof}

\noindent
The next goal is to find an expression for $\E\l[ T_{2}-T_{2}' | T_{2} \r]$ that allows us to define $F(T_2, T_2')$ and $f(T_2)$. First observe that $\E\left[ T_2'\left| T_2\right. \right]= \E\left[\E \left[  T_2'\left| \sigma,\pi\right. \right]\left|T_2 \right.\right]$, and that one also has
\begin{align*}{\label{expX'}} \nonumber
    \E \l[  T_2'\big| \sigma,\pi \r] &= \E \l[ T_2- \sum_{j=1}^{3} a_{I_{j},\sigma(I_{j}),\pi(I_{j})}+ 
    \frac{1}{2}\l(\sum_{j=1}^{3} a_{I_{j},\sigma\circ \tau_{1}(I_1,I_{2},I_{3})(I_{j}),\pi\circ \tau_{2}(I_1,I_{2},I_{3})(I_{j})} \r) \r.\\ \nonumber
    &+ \l.\frac{1}{2}\l(\sum_{j=1}^{3} a_{I_{j},\sigma\circ \tau_{2}(I_1,I_{2},I_{3})(I_{j}),\pi\circ \tau_{1}(I_1,I_{2},I_{3})(I_{j})}\r) \big| \sigma,\pi\r] .
\end{align*}
We deal separately with the terms in the last expression. 
First of all,
\begin{equation*}
     \E\l[\E\l[ \sum_{j=1}^{3} a_{I_{j},\sigma(I_{j}),\pi(I_{j})} \l| \sigma,\pi\r.\r] \big| T_{2}\r] = \E\l[\sum_{j=1}^3 \frac1n\sum_{i=1}^n a_{i, \sigma(i), \pi(i)} \big| T_{2}\r] = \frac{3}{n} T_{2}.
\end{equation*}
Now, for any $j \in \{1,2,3\}$, 
\begin{align*}
&\E\l[a_{I_{j},\sigma\circ \tau_{1}(I_1,I_{2},I_{3})(I_{j}),\pi\circ \tau_{2}(I_1,I_{2},I_{3})(I_{j})}\l| \sigma, \pi \r.\r] =  \E\l[a_{I_{j},\sigma\circ \tau_{2}(I_1,I_{2},I_{3})(I_{j}),\pi\circ \tau_{1}(I_1,I_{2},I_{3})(I_{j})}\l| \sigma, \pi \r.\r]\\
&\quad = \frac{1}{n(n-1)(n-2)} \sum_{(i,j,k) \in C_{3}} a_{i,\sigma(j),\pi(k)}
\end{align*}
 which implies that
\begin{align*}
&\E\l[\frac{1}{2}\l(\sum_{j=1}^{3} a_{I_{j},\sigma\circ \tau_{1}(I_1,I_{2},I_{3})(I_{j}),\pi\circ \tau_{2}(I_1,I_{2},I_{3})(I_{j})} \r) + \frac{1}{2}\l( \sum_{j=1}^{3} a_{I_{j},\sigma\circ \tau_{2}(I_1,I_{2},I_{3})(I_{j}),\pi\circ \tau_{1}(I_1,I_{2},I_{3})(I_{j})}\r) \big| \sigma,\pi\r] \\
&= \frac{3}{n(n-1)(n-2)}\sum_{(i,j,k) \in C_{3}} a_{i,\sigma(j),\pi(k)}\\
&= \frac{3}{n(n-1)(n-2)}\l[ \sum_{i=1}^{n}\sum_{j=1}^{n}\sum_{k=1}^{n} a_{i,j,k} - \sum_{i} a_{i,\sigma(i),\pi(i)}-\sum_{(i,j,k) \in C_2} a_{i,\sigma(j),\pi(k)}\r].
\end{align*}
We define $Y(\sigma,\pi) := \sum_{(i,j,k) \in C_2} a_{i,\sigma(j),\pi(k)}$, and notice that $0\le Y(\sigma,\pi) \le 3n(n-1)$ irrespective of $\sigma$ and $\pi$. Putting the previous pieces together, we then get
\begin{equation}{\label{bound on T'}}
\E\l[T_2' | T_2\r] = T_2 - \frac{3}{n} T_{2} + \frac{3n \E[T_{2}]}{(n-1)(n-2)}-\frac{3 T_{2} + 3 \E[Y(\sigma,\pi)|T_{2}]}{n(n-1)(n-2)}.
\end{equation}
Using the expressions above, we have
\begin{align*}
         \E\l[ T_{2}-T_{2}' | T_{2} \r] &= \l(\frac{3}{n}+ \frac{3}{n(n-1)(n-2)}\r)\cdot T_{2}- \frac{3n  \E[T_{2}]}{(n-1)(n-2)} + \frac{3 \E[Y(\sigma,\pi)|T_{2}]}{n(n-1)(n-2)}\\
         &=  \frac{3(n^2-3n+3) T_{2}}{n(n-1)(n-2)} -  \frac{3n \E[T_{2}]}{(n-1)(n-2)}  + \frac{3\E[Y(\sigma,\pi)|T_{2}]}{n(n-1)(n-2)}\\
         &= \frac{3(n^{2}-3n+3)}{n(n-1)(n-2)}\cdot\l[T_{2}-\E[T_{2}]\r] - \frac{9 \E[T_{2}]}{n(n-2)}+ \frac{3 \E[Y(\sigma,\pi)|T_{2}]}{n(n-1)(n-2)}
\end{align*}
and it makes sense now to define
$$F(T_{2},T_{2}')=\frac{n(n-1)(n-2)}{3(n^{2}-3n+3)}\cdot(T_{2}-T_{2}')$$ 
so that 
\begin{equation}\label{expF(X,X')}
        f(T_{2}) = T_{2}- \E[T_{2}] - \frac{3(n-1)\E[T_{2}]}{n^{2}-3n+3}+ \frac{\E\l[ Y(\sigma,\pi)\l|T_{2} \r.\r]}{n^{2}-3n+3} .
\end{equation}  
From (\ref{expF(X,X')}) one has
$$f(T_{2})-f(T_{2}') = T_{2}-T_{2}' + \frac{1}{n^{2}-3n+3} \l[\E\l[ Y(\sigma,\pi)\l|T_{2} \r.\r] -  \E\l[ Y(\sigma',\pi')\l|T_{2}' \r.\r]  \r]  $$
and the following important consequences
\begin{equation}{\label{fgreat}}
f(T_{2})\ge T_{2}-\E[T_{2}] - 3+O\l(\frac{1}{n}\r) 
\end{equation} 
and 
\begin{equation}{\label{fsmall}}
f(T_{2})\le T_{2}-\E[T_{2}] + 3 +O\l(\frac{1}{n}\r). 
\end{equation}
As before we want to upper bound the function $v(T_{2})$, to finally invoke Lemma \ref{lem:vx}.
\begin{align} \nonumber
        v(T_{2}) &= \frac12 \E\l[|(f(T_{2}) - f(T_{2}'))\cdot F(T_{2}, T_{2}')| \ \bigg| \ T_{2}\r] \\ 
         &\le \frac{n(n-1)(n-2)}{6(n^{2}-3n+3)}\E\l[(T_{2}-T_{2}')^{2}\l|T_{2} \r.  \r] + 
         \frac{n^2(n-1)^{2}(n-2)}{2(n^{2}-3n +3)^2}\E\l[ \l| T_{2}-T_{2}' \r|\l|T_{2}  \r. \r]  \label{expV(X)}
\end{align}
So the last task is now to upper bound these two quantities. 
With a calculation similar as the one used to obtain (\ref{bound on T'}), we get
\begin{align*}
        &\E\l[\l|T_{2}-T_{2}'\r| |\ T_{2} \r] \le \E\l[ \sum_{j=1}^{3}a_{I_{j},\sigma(I_{j}),\pi(I_{j})} \bigg|\ T_{2} \r] \\
        &\quad + \frac{1}{2} \E\l[\sum_{j=1}^{3} a_{I_{j},\sigma \circ \tau_{1}(I_{1},I_{2},I_{3})(I_{j}),\pi \circ \tau_{2}(I_{1},I_{2},I_{3})(I_{j})} + \sum_{j=1}^{3} a_{I_{j},\sigma \circ \tau_{2}(I_{1},I_{2},I_{3})(I_{j}),\pi \circ \tau_{1}(I_{1},I_{2},I_{3})(I_{j})}\ \bigg|\ T_{2} \r]\\
        &= \frac{3}{n} T_{2} + \frac{3n}{(n-1)(n-2)} \E[T_{2}]-\frac{3}{n(n-1)(n-2)}\l[T_{2}+ \E\l[Y(\sigma,\pi) | T_{2}\r]\r]\\
        & \le \frac{3}{n} T_{2} + \frac{3n}{(n-1)(n-2)} \E[T_{2}].
\end{align*}
and
\begin{align*}
    &\E\l[\l( T_{2}-T_{2}'\r)^{2}\l| T_{2} \r. \r] = \E\l[\frac{1}{2}\l(\sum_{j=1}^{3} a_{I_{j},\sigma(I_{j}),\pi(I_{j})}- \sum_{j=1}^{3} a_{I_{j},\sigma \circ \tau_{1}(I_{1},I_{2},I_{3})(I_{j}),\pi \circ \tau_{2}(I_{1},I_{2},I_{3})(I_{j})} \r)^2 \r.\\ 
        &\quad + \l. \frac{1}{2} \l(\sum_{j=1}^{3} a_{I_{j},\sigma(I_{j}),\pi(I_{j})}- \sum_{j=1}^{3} a_{I_{j},\sigma \circ \tau_{2}(I_{1},I_{2},I_{3})(I_{j}),\pi \circ \tau_{1}(I_{1},I_{2},I_{3})(I_{j})} \r)^2 \bigg| T_{2} \r]\\
    &\le 3 \E\l[ \sum_{j=1}^{3}a_{I_{j},\sigma(I_{j}),\pi(I_{j})}+ \frac{1}{2}\l(\sum_{j=1}^{3} a_{I_{j},\sigma \circ \tau_{1}(I_{1},I_{2},I_{3})(I_{j}),\pi \circ \tau_{2}(I_{1},I_{2},I_{3})(I_{j})} + \sum_{j=1}^{3} a_{I_{j},\sigma \circ \tau_{2}(I_{1},I_{2},I_{3})(I_{j}),\pi \circ \tau_{1}(I_{1},I_{2},I_{3})(I_{j})} \r) \bigg| T_{2} \r] \\
    &\le \frac{9}{n}T_{2}+  \frac{9n}{(n-1)(n-2)} \E[T_{2}].
\end{align*}
Using these estimates in (\ref{expV(X)}), together with the lower bound on $f(X)$ obtained in (\ref{fgreat}), we have 
\begin{align*}
v(T_{2}) &\le \l(3+O\l(\frac{1}{n}\r)\r)(T_{2}+ \E[T_{2}]) \\
&\le \l(3+O\l(\frac{1}{n}\r)\r) \l(T_{2}-\E[T_{2}]-3 + 2\E[T_{2}]+3+O\l(\frac{1}{n} \r)\r) \\
&\le \l(3+O\l(\frac{1}{n}\r)\r)\l(f(T_{2})+ 2\E[T_{2}] + 3 + O\l(\frac{1}{n} \r) \r).     
\end{align*}
So, making use again of Lemma \ref{lem:vx}, we obtain the concentration inequality
$$
\P\left[|f(T_{2})| \geq t \right] \le 2 \exp\left( -\frac{t^{2}}{12\E(X) + 18 + 6(1 + O(1/n))\cdot t} \right) .
$$
By (\ref{fgreat}) and (\ref{fsmall}) we obtain a new concentration inequality:
\begin{align}{\label{concT2}} \nonumber
\P(|T_{2} - \E[T_{2}]| \geq t) &\leq \P(|f(T_{2})| + 3 + O(1/n) > t) \\ 
&= \P(|f(T_{2})| \geq t - 3 + O(1/n)) \\ \nonumber
&\leq 2 \exp\left( -\frac{(t-3 +O(1/n))^{2}}{12\E(T_{2}) + 18 + 6(1+O(1/n))(t-3)} \right).
\end{align}\hfill{$\square$}

\section{Conclusion and Possible Extension}

We obtained a Bernstein type tail bounds for the statistics $T_1$ and $T_2$. These are three dimensional generalizations of the sum of choices $T_3$, which has already been studied in \cite{Cha05} and \cite{Hoe51}.
A natural extension of this work consists in finding a concentration inequality for a d-dimensional generalization of the statistic $T_2$, of the form
$$T_4 = \sum_{i} a_{i, \pi_1(i),...,\pi_{d-1}(i)} $$
where $a_{i_1, ..., i_d} \in [0, 1]$ and $\pi_j \in \S_n$ for $j = 1, ..., d-1$. Here, one should extract $d$ indexes $(I_1, ..., I_d)$ \textit{without replacement} and consider the functions $\tau_{j} : [n]^d \to \S_n$ for $j = 1, ..., d-1$, such that $\tau_j(I_1,..., I_d)$ is one of the $d-1$ cyclic permutations which are not the identity (for clarity, let it be the rotation of $j$ positions forward, where $I_1$ gets mapped into $I_{j+1}$). When considering $d$ indexes, it is equivalent to rotate each index by $j$ positions in one direction, or by $d-1-j$ positions in the opposite direction. For this reason one has that $\tau_j(I_1,..., I_d)^{-1} = \tau_{d-1-j}(I_1,..., I_d)$.
We define 
$$\pi_j' = \l\{\begin{array}{ll}
\pi_j \circ \tau_j(I_1, ..., I_d) \qquad &\text{with probability } \ \frac12 \\[3pt]
\pi_j \circ \tau_j(I_1, ..., I_d)^{-1} \qquad &\text{with probability } \ \frac12 \end{array}\r.$$
and then proceed as done before, defining $T_4'$ using the $\pi'_j$ permutations and showing that $(T_4, T_4')$ is an exchangeable pair. To find the tail bound for $T_{4}$, the calculation is similar as before but more cumbersome, and we have not implemented it in the current paper. The main reason for choosing this type of cyclic permutation is because, for every fixed $l \in [d]$, the indexes $\left(I_{l}, \tau_{1}(I_{1},\ldots,I_{d})(I_{l}), \ldots, \tau_{d}(I_{1},\ldots, I_{d})(I_{l}) \right)$ are distinct. Then, when fixing $\pi_{1},\ldots, \pi_{d-1}$, the expectation 
$$
\E\left[  a_{I_{l}, \pi_{1} \circ \tau_{1}(I_{1},\ldots,I_{d})(I_{l}), \ldots, \pi_{d-1} \circ \tau_{d-1}(I_{1},\ldots,I_{d})(I_{l})  }\left| \pi_{1},\ldots, \pi_{d-1} \right. \right]
$$
 contains the sum over all the possible independent directions, while leaving out the cases when some indexes are repeated. This is the fundamental observation which allows us to write $\E\left[ T_{4} - T_{4}'\left| T_{4} \right.  \right]$ in the convenient way to apply Lemma \ref{lem:vx}.

\section{Acknowledgements}
The authors are pleased to thank J. Michael Steele for his encouragement and advice.

\end{document}